\documentclass[11pt]{article}

\usepackage{amsmath,amssymb,amsthm}
\usepackage[usenames,dvipsnames]{color}
\usepackage{float}
\usepackage{hhline}
\usepackage{arydshln}
\usepackage{multicol}
\usepackage{ulem}
\usepackage{verbatim}
\usepackage{lmodern}
\usepackage{wrapfig}
\usepackage{setspace} 
\usepackage{hyperref}
\usepackage{graphicx, epstopdf}
\usepackage[caption=false]{subfig}
\usepackage{algorithm, algpseudocode}
\usepackage{fancybox}
\usepackage{lmodern}
\usepackage{sidecap}
\usepackage{pdfpages}
\usepackage{xspace}

\newcommand{\fast}{\texttt{SCvx-fast}\xspace}
\newcommand{\scvx}{\texttt{SCvx}\xspace}
\newcommand{\pac}{\textit{project-and-convexify}\xspace}

\newcommand{\orderof}{o}

\newcommand{\real}{\mathbb{R}}

\newtheorem{theorem}{Theorem}[section]
\newtheorem{lemma}{Lemma}[section]

\theoremstyle{definition}

\newtheorem{assumption}{Assumption}[section]
\newtheorem{problem}{Problem}[section]

\theoremstyle{remark}
\newtheorem*{remark}{Remark}

\title{\fast: A Superlinearly Convergent Algorithm for A Class of Non-Convex Optimal Control Problems}
\author{
        Yuanqi Mao, Beh\c cet A\c c\i kme\c se \\
                William E. Boeing Department of Aeronautics and Asgtronautics\\
        University of Washington, 
        Seattle, WA 98105
}
\date{\today}

\begin{document}
\maketitle

\begin{abstract}
In this paper, we extend the results from \cite{mao2017successive}, and formally propose the \fast algorithm, a new addition to the Successive Convexification algorithmic framework. The said algorithm solves non-convex optimal control problems with specific types of state constraints (i.e. union of convex keep-out zones) and is faster to converge than \scvx, its predecessor. In order to preserve more feasibility, the proposed algorithm uses a novel \pac procedure to successively convexify both state constraints and system dynamics, and thus a finite dimensional convex programming subproblem is solved at each succession. It also gets rid of the dependency on trust regions, gaining the ability to take larger steps and thus ultimately attaining faster convergence. The extension is in three folds as follows. i) We can now initialize the algorithm from an infeasible starting point, and regain feasibility in just one step; ii) We get rid of the smoothness conditions on the constraints so that a broader range of ``obstacles" can be included. Significant changes are made to adjust the algorithm accordingly; iii) We obtain a proof of superlinear rate of convergence, a new theoretical result for \fast. Benefiting from its specific problem setup and the \pac procedure, the \fast algorithm is particularly suitable for solving trajectory planning problems with collision avoidance constraints. Numerical simulations are performed, affirming the fast convergence rate. With powerful convex programming solvers, the algorithm can be implemented onboard for real-time autonomous guidance applications.
\end{abstract}

\section{Introduction} \label{sec:intro}

Non-convex optimal control problems emerge in a broad range of science and engineering disciplines. Finding a global solution to these problems is generally considered NP-hard. Heuristics like simulated annealing, \cite{bertsimas2010robust}, or combinatorial methods like mixed integer programming, \cite{richards2002}, can compute globally optimal solutions for special classes of problems. In many engineering applications however, finding a local optimum or even a feasible solution with much less computational effort is a more favorable route. This is particularly the case with real-time control systems, where efficiency and convergence guarantees are more valuable than optimality. An example in aerospace applications is the planetary landing problem, see~\cite{pointing2013, lars_sys12, steinfeldt2010guidance}. Non-convexities in this problem include minimum thrust constraints, nonlinear gravity fields and nonlinear aerodynamic forces. State constraints can also render the problem non-convex. A classic example is imposing collision avoidance constraints. For instance, \cite{accikmese2006convex} discusses the collision avoidance in formation reconfiguration of spacecraft, \cite{augugliaro2012generation} considers the generation of collision-free trajectories for a quad-rotor fleet, and~\cite{liu2014solving} study the highly constrained rendezvous problem.

Given the complexity of such non-convex problems, traditional Pontryagin's maximum principle-based approaches, e.g.~\cite{rockafellar1972state}, can fall short. On the other hand, directly applying optimization methods to solve the discretized optimal control problems has gained in popularity thanks to algorithmic advancements in nonlinear programming, see e.g.~\cite{hull1997, buskens}. However, general nonlinear optimization can sometimes be intractable in the sense that a bad initial guess could potentially lead to divergence, and also there are few known bounds on the computational effort needed to reach optimality. This makes it difficult to implement for real-time or mission critical applications because they cannot afford either divergence or a heavy load of computation. Convex optimization, on the other hand, can be reliably solved in polynomial time to global optimality, see e.g.~\cite{BoydConvex}. More importantly, recent advances have shown that these problems can be solved in real-time by both generic Second Order Cone Programming (SOCP) solvers, e.g.~\cite{domahidi2013ecos}, and by customized solvers which take advantage of specific problem structures, e.g.~\cite{mattingley2012, dueri2014automated}. This motivates researchers to formulate optimal control problems in a convex programming framework for real-time purposes, e.g., real-time Model Predictive Control (MPC), see~\cite{houska2011auto, Zeilinger2014683}.

In order to take advantage of these powerful convex programming solvers, one crucial step is to convexify the originally non-convex problems. Recent results on a technique known as \textit{lossless convexification}, e.g.~\cite{behcet_aut11, lars_sys12, matt_aut14} have proven that certain types of non-convex control constraints can be posed as convex ones without introducing conservativeness. \cite{liu2015entry} also gives a result on convexification of control constraints for the entry guidance application. For nonlinear dynamics and non-convex state constraints, collision avoidance constraints in particular, one simple solution is to perform query-based collision checks, see \cite{allen2016real}. However, to be more mathematically tractable, \cite{augugliaro2012generation, schulman2014motion,  chen2015decoupled} propose to use (variations of) sequential convex programming (SCP) to iteratively convexify non-convexities. While these methods usually perform well in practice, no convergence results have been reported yet. As an effort to tackle this open challenge, \cite{SCvx_cdc16} propose an successive convexification (\scvx) algorithm that successively convexifies the dynamics with techniques like \textit{virtual control} and \textit{trust regions}, and more importantly, give a convergence proof of that algorithm.

To include state constraints in the \scvx algorithmic framework, a few enhancements need to be made. \cite{hauser2006barrier} relax the state constraints by using a barrier function, but do not provide theoretical guarantees. In this paper, we relax the equations of system dynamics into inequalities by using an exact penalty function, and then we propose a \pac procedure to convexify both the state constraints and the relaxed system dynamics. While introducing conservativeness is inevitable in the process, the proposed algorithm preserves much more feasibility than similar results in~\cite{rosen1966iterative, liu2014solving,lipp2016variations}. Then, under some mild assumptions, we present a global convergence proof, which not only guarantees that the algorithm will converge, but also demonstrates that the convergent solution recovers local optimality for the original problem. Finally, we give a proof of the algorithm's superlinear convergence rate, followed by numerical evidence supporting both claims.

One clear advantage of the \fast algorithm proposed in this paper is that this algorithm does not have to resort to trust regions, as in e.g.~\cite{SCvx_cdc16,szmuk2016successive,szmuk2017successive}, to guarantee convergence. This property allows the algorithm to potentially take a large step in each succession, thereby greatly accelerating the convergence process, which is exactly the case shown by the numerical simulations. It also worth noting that the proposed algorithm only uses the Jacobian matrix, i.e. first-order information; therefore, we do not have to compute Hessian matrices, otherwise that task itself could be computationally expensive. To the best of our knowledge, the main contributions of this work are:
\begin{itemize}
	\item Proposition of \fast, an variant of \textit{Successive Convexification} algorithm with a \pac procedure handling both nonlinear dynamics and non-convex state constraints.
	\item A global convergence proof with local optimality recovery, with significant adjustments made to accommodate nonsmooth  constraints.
	\item A superlinear convergence rate proof supporting it real-time applicability from a theoretical perspective.
\end{itemize}

The remainder of this article is organized as follows.
Section~\ref{sec:prob} describes the optimal control problem we aim to solve.
The proposed \fast algorithm is described in Section~\ref{sec:algo}, and its convergence and rate of convergence are analyzed in Section~\ref{sec:conv}.
Section~\ref{sec:num} presents an illustrative numerical example.
Finally, Section~\ref{sec:conc} gives the conclusions.


\section{Problem Formulation} \label{sec:prob}
In this paper, we consider the following discrete-time optimal control problem:
\begin{problem} \label{prob:fast_ori}
	Discrete-time Optimal Control Problem with Special Constraints
	\begin{subequations}
		\begin{equation}
		\min \quad J(x_{i},u_{i}):= \sum_{i=1}^{T}\phi(x_{i},u_{i}), \label{cost:original}
		\end{equation}
		\quad subject to
		\begin{align}
		&x_{i+1}-x_{i}=f(x_{i},u_{i}) &i=1,2,T-1, \label{eq:fast_dynamics} \\
		&h(x_{i})\geq 0 &i=1,2,T, \label{eq:fast_state_con} \\
		&u_{i}\in U_{i} \subseteq \real^{m} &i=1,2,T-1, \\
		&x_{i}\in X_{i} \subseteq \real^{n} &i=1,2,T.
		\end{align}
	\end{subequations}
\end{problem}

Here, $ x_{i},u_{i} $ represent discrete state/control at each temporal point, $ T $ denotes the final time, and $ X_{i},U_{i} $ are assumed to be convex and compact sets. We also assume that the objective function in~\eqref{cost:original} is continuous and convex, as is the case in many optimal control applications. For example, the minimum fuel problem has $ \phi(x_{i},u_{i}) = \|u_{i}\| $, and the minimum time problem has $ \phi(x_{i},u_{i}) = 1 $. 
Equation~\eqref{eq:fast_dynamics} represent the system dynamics, where $ f(x_{i},u_{i})\in \real^{n} $ is, in general, a nonlinear function that is at \textbf{continuous}.
Equation~\eqref{eq:fast_state_con} are the additional state constraints, where $ h(x_{i})\in \real^{s} $ is also \textbf{continuous}, and could be nonlinear as well. Note that we do not impose non-convex control constraints here because we can leverage lossless convexification, see e.g.~\cite{behcet_aut11}, to convexify them beforehand. Finally, note that ~\eqref{eq:fast_dynamics} and~\eqref{eq:fast_state_con} render the problem non-convex.

To reflect the specialty of \texttt{SCvx-fast}, we need a few assumptions on $ f(x_{i},u_{i}) $ and $ h(x_{i}) $:
\begin{assumption} \label{asup:convex_f}
	$ f_{j}(x_{i},u_{i}) $ is a convex function over $ x_{i} $ and $ u_{i} $, $ \forall j = 1,2,...n $.
\end{assumption}
In fact, a wide range of optimal control applications, for example, systems with double integrator dynamics and aerodynamic drag (constant speed), satisfy Assumption~\ref{asup:convex_f}. It also includes all linear systems. Similarly, we also have
\begin{assumption} \label{asup:convex_h}
	$ h_{j}(x_{i}) $ is a convex function over $ x_{i} $, $ \forall j = 1,2,...s $.
\end{assumption}
An example for Assumption~\ref{asup:convex_h} is collision avoidance constraints, where the shape of each keep-out zone is convex or can be convexly decomposed. See Figure~\ref{fig:keep_out} for a simple illustration. Note that at this stage, these convexity assumptions do not change the non-convex nature of the problem.
\begin{figure}[!h]
	\begin{center}
		\includegraphics[width=0.4\textwidth]{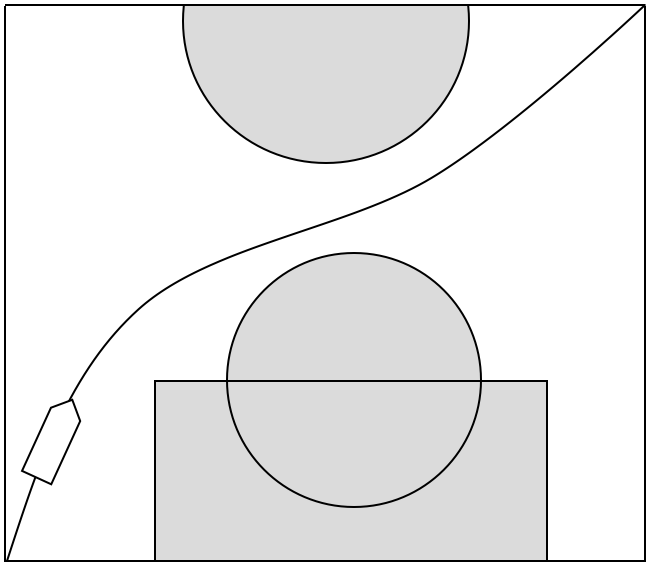}
		\caption{Convex shaped keep-out zone as state constraints} 
		\label{fig:keep_out}
	\end{center}
\end{figure}

To convert the optimal control problem into a finite dimensional optimization problem, we treat state variables $ x_{i} $ and control variables $ u_{i} $ as a single variable
$$ y = \left( x_{1}^{T},..., x_{T}^{T}, u_{1}^{T},..., u_{T-1}^{T} \right)^{T} \in \real^{N}, $$
where $ N = m(T-1)+nT $. Let $ Y $ be the Cartesian product of all $ X_{i} $ and $ U_{i} $, then $ y \in Y $. $ Y $ is a convex and compact set because $ X_{i} $ and $ U_{i} $ are. In addition, we let $ g_{i}(x_{i}, u_{i}) = f(x_{i}, u_{i}) - x_{i+1} + x_{i} $, and $ g(y) = \left( g_{1}^{T}, g_{2}^{T},..., g_{T-1}^{T} \right)^{T} \in \real^{n(T-1)} $. Note that each component of $ g(y) $ is convex over $ y $ by Assumption~\ref{asup:convex_f}, and the dynamic equation~\eqref{eq:fast_dynamics} becomes $ g(y) = 0 $. We also let $ h(y) = \left( h(x_{1})^{T}, h(x_{2})^{T},..., h(x_{T})^{T} \right)^{T} \in \real^{sT} $, then each component of $ h(y) $ is convex over $ y $ by Assumption~\ref{asup:convex_h}, and~\eqref{eq:fast_state_con} becomes $ h(y)\geq 0 $. In summary, we have the following non-convex optimization problem:
\begin{equation} \label{eq:ori_prob}
J(y^{*}) = \underset{y}{min}\{J(y)\ |\ y\in Y, g(y)=0, h(y)\geq 0 \}.
\end{equation}
By leveraging the theory of exact penalty methods, see e.g.~\cite{han1979exact},~\cite{SCvx_cdc16}, we move $ g(y)=0 $ into the objective function without compromising optimality:
\begin{theorem}[\textbf{Exactness}] \label{thm:fast_exactness}
	Let $ P(y) = J(y)+\lambda \|g(y)\|_{1} $ be the penalty function, and $ \bar{y} $ be a stationary point of
	\begin{equation} \label{eq:pen_prob}
	\underset{y}{min}\{P(y)\ |\ y\in Y, g(y)\geq 0, h(y)\geq 0 \}
	\end{equation}
	with Lagrangian multiplier $ \bar{\mu} $ for equality constraints. Then, if the penalty weight $ \lambda $ satisfies $ \lambda\geq \|\bar{\mu}\|_{\infty} $, and if $ \bar{y} $ is feasible for~\eqref{eq:ori_prob}, then $ \bar{y} $ is a critical point of~\eqref{eq:ori_prob}.
\end{theorem}
Since each component of $ g(y) $ is convex, and $ \|\cdot \|_{1} $ is convex and nondecreasing due to the constraint $ g(y)\geq 0 $, then $ P(y) $ is a convex function by the composition rule of convex functions, see~\cite{BoydConvex}. This marks our first effort towards the convexification of~\eqref{eq:ori_prob}.

Let $ q(y) = \left( g(y)^{T}, h(y)^{T}\right)^{T} \in \mathbb{R}^{M} $, where $ M = sT+n(T-1) $, then we may rewrite~\eqref{eq:pen_prob} as
\begin{equation} \label{eq:final_prob}
\underset{y}{min}\{P(y)\ |\ y\in Y, q(y)\geq 0 \}
\end{equation}
where $ P(y) $ is continuous and convex, and each component of $ q(y) $ is a convex function. By doing this, we are essentially treating constraints due to dynamics as another keep-out zone. Denote
\begin{equation*}
F = \{y\ |\ y\in Y, q(y)\geq 0 \}
\end{equation*}
as the feasible set. Note that $ F $ is compact but not convex.

\section{The \texttt{SCvx-fast} Algorithm} \label{sec:algo}

\subsection{Infeasible Initialization}
Previously in~\cite{mao2017successive}, we require a feasible starting point. Now we propose a preprocessing routine that linearize the violating constraints at the starting point, which will get us out of the infeasible region in just one step. See Algorithm~\ref{algo:infeas-ini} for details and Figure~\ref{fig:infeasible} for an geometric illustration of this procedure (in 2-D case).
\begin{algorithm}[!h]
	\caption{The Infeasible Initialization Routine}
	\label{algo:infeas-ini}
	\begin{algorithmic}[1]
		\Procedure{\texttt{Infeasible-Initialization}}{$ z^{(0)} $}
		\State \textbf{input} Any point $ z^{(0)} \notin F $.
		\State \textbf{step 1} Approximate any intersecting constraints with their minimum volume ellipsoidal cover as the new constraint.
		\State \textbf{step 2} Identify any infeasible constraint $ h_i(z_{i}^{(0)}) < c $, where $ c $ represents all the constant parts of that constraint.
		\State \textbf{step 3} Linearize $ h_i(y) $ with respect to $ z_{i}^{(0)} $ as $ l_{z_{i}^{(0)}}(y) $.
		\State \textbf{step 4} Project $ z^{(0)} $ onto the half-space intersected by all the $ l_{z_{i}^{(0)}}(y) = c $ to get a new point $ \tilde{z}_{i}^{(0) } $.
		\State \textbf{return} $ \tilde{z}_{i}^{(0)} $ as the new starting point.
		\EndProcedure
	\end{algorithmic}
\end{algorithm}

\begin{figure}[!h]
	\begin{center}
		\includegraphics[width=0.5\textwidth]{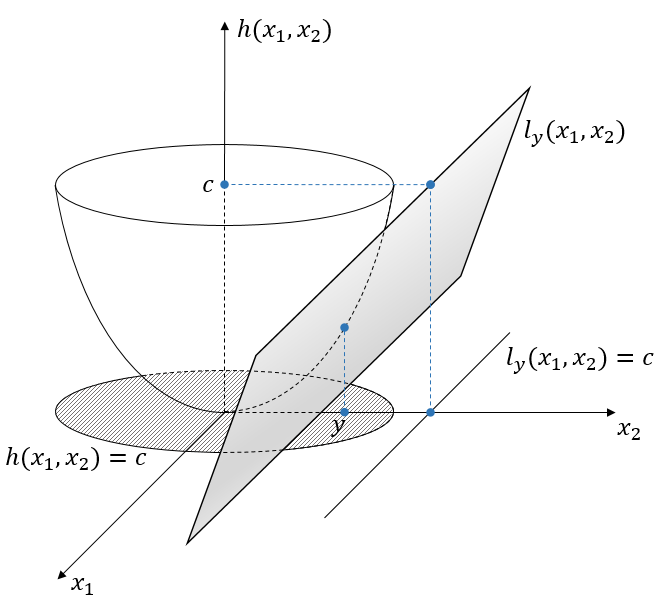}
		\caption{Initialize the algorithm at an infeasible point $ y $. The first linearization will get us to the feasible region $ l_y > c $.}
		\label{fig:infeasible}
	\end{center}
\end{figure}

\begin{theorem}[\textbf{Feasibility of $ \tilde{z}^{(0)} $}]
	Given Assumption~\ref{asup:convex_h}, Algorihtm~\ref{algo:infeas-ini} produces a feasible initial trajectory in terms of constraints, i.e. $ \tilde{z}^{(0)} \in F $.
\end{theorem}
\begin{proof}
	By the convexity of $ h(y) $ from Assumption~\ref{asup:convex_h}, we know that its epigraph $ epi(h) = \{ (y,w) \, | \, y \in Y, w \in \real, h(y) \leq w \} $ is a convex set. Denote the projection point of $ z^{(0)} $ onto $ epi(h) $ as $ z_{p} $, then $ l_{z^{(0)}}(y) $ is the supporting hyperplane of $ epi(h) $ at $ z_{p} $. Thanks to the separating hyperplane theorem \cite{BoydConvex}, we have $ l_{z^{(0)}}(y)  = c $ outside of $ epi(h) $, thus $ \tilde{z}^{(0)} $, as the projection of $ z^{(0)} $, will enjoy the property $ h(\tilde{z}^{(0)}) \geq c $, i.e. $ \tilde{z}^{(0)} \in F $.
\end{proof}

\subsection{Project-and-Convexify}
Next we will introduce the \emph{project-and-convexify} procedure. For any point $ z \in F $, let $ \tilde{\triangledown}_{y} q(z) $ be the generalized Jacobian matrix of $ q(y) $ evaluated at $ z $. Now, if we directly linearize $ q(y) $ at $ z $ as in~\cite{liu2014solving}, there might be a gap between the linearized feasible region and $ F $ since $ z $ could be in the interior of $ F $. The gap will increase as $ z $ moves further away from the boundary $ q(y)=0 $. This is not a desirable situation, because a fairly large area of the feasible region is not utilized. In other words, we introduced artificial conservativeness. To address this issue, we first introduce a \textit{projection} step, which essentially projects $ z $ onto each constraint, and obtains each projection point. Then, we linearize each constraint at its own projection point. See Figure~\ref{fig:project} for an illustration in $ \mathbb{R}^{2} $.
\begin{figure}[!h]
	\begin{center}
		\includegraphics[width=0.5\textwidth]{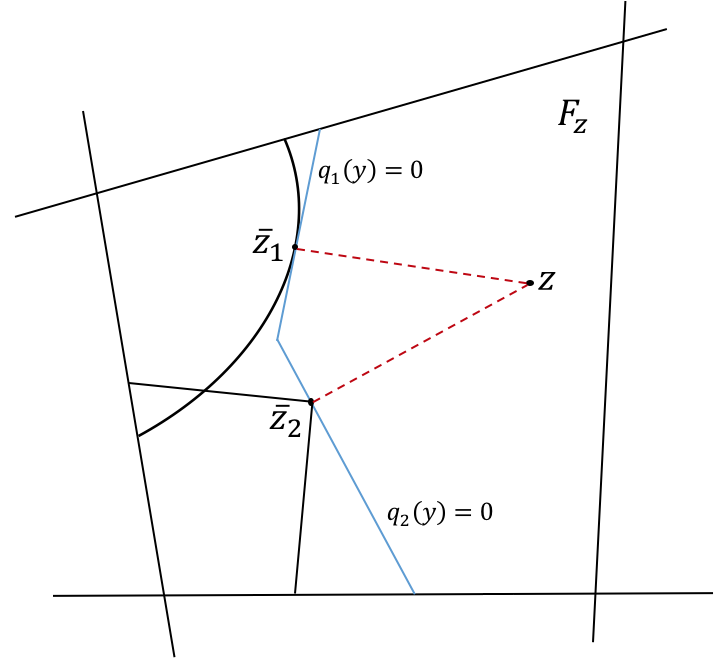}
		\caption{Project the current point $ z $ onto each constraint. Note that when the constraint is non-smooth, we pick the supporting hyperplane that is orthogonal to $ z - \bar{z} $ as our linearization.}
		\label{fig:project}
	\end{center}
\end{figure}

To formalize, let $ q_{j}(y), j=1,2...,M $ represent each component of $ q(y) $, i.e. each constraint. Note that $ q_{j}(y) $ is a convex function, hence $ q_{j}(y) \leq 0 $ is a closed convex set. Using the well-known Hilbert projection theorem, see e.g.\cite{wulbert_project}, we have
\begin{theorem}[\textbf{Uniqueness of the projection}]
	For any $ z \in F $, there exists a unique point
	\begin{equation} \label{eq:project}
	\bar{z}_j = \underset{y}{argmin} \left\lbrace \|z-y\|_{2}\ |\ q_{j}(y) \leq 0 \right\rbrace,
	\end{equation}
	called the projection of $ z $ onto $ q_{j}(y) \leq 0 $.
\end{theorem}
Equation~\eqref{eq:project} is a simple convex program of low dimension that can be solved quickly (sub-milliseconds) using any convex programming solver. Alternatively, for some special convex sets (e.g. cylinders), \eqref{eq:project} can be solved analytically, which is even faster. Doing this for each constraint, we obtain a set of projection points, $ \{\bar{z}_{1}, \bar{z}_{2},..., \bar{z}_{M} \} $. Note that these projection points must lie on the boundary of $ q_{j}(y) \leq 0 $, i.e.
$
q_{j}(\bar{z}_{j}) = 0, \; \forall \; j = 1,..., M.
$
For a fixed $ z \in F $, let $ l_{j}(y,z) $ be the linear approximation of $ q_{j}(y) $:
\begin{equation} \label{eq:linearization}
l_{j}(y,z) = \tilde{\triangledown}_{y}q_{j}(\bar{z}_{j})(y - \bar{z}_{j}),
\end{equation}
and let $ l(y,z) = (l_{1}^{T}, l_{2}^{T},..., l_{M}^{T})^{T} \in \mathbb{R}^{M} $. Not that here when the constraint function $ q_{j} $ is not differentiable everywhere, $ \tilde{\triangledown}_{y}q_{j}(\bar{z}_{j}) $ denote the gradient to the supporting hyperplane orthogonal to $ z - \bar{z}_{j} $. For each $ z \in F $, denote
\begin{equation*}
F_{z} = \{y\ |\ y\in Y,\, l(y,z)\geq 0 \}
\end{equation*}
as the feasible region after linearization. $ F_{z} $ also defines a point-to-set mapping, $ F_{z}: z \rightarrow F_{z} $.
Note that each component of $ l(y,z)\geq 0 $, i.e. $ l_{j}(y,z)\geq 0 $ represents a half-space. Hence $ l(y,z)\geq 0 $ is the intersection of half-spaces, which means $ F_{z} $ is a convex and compact set.

\begin{remark}
	Convexification of $ F $ by using $ F_{z} $ also inevitably introduces conservativeness, but one can verify that it is the best we can do to maximize feasibility while preserving convexity.
\end{remark}

The following lemma gives an invariance result regarding the point-to-set mapping $ F_{z} $. It is essential to our subsequent analyses.
\begin{lemma}[\textbf{Invariance of $ F_{z} $}] \label{lem:invariance}
	For each $ z \in F$, we have $ z \in F_{z} \subseteq F. $
\end{lemma}
\begin{proof}
	For each $ z \in F $, and $ \forall j = 1,2,..., M $, from~\eqref{eq:linearization}, we have
	\begin{equation*}
	l_{j}(z,z) = \tilde{\triangledown}_{y}q_{j}(\bar{z}_{j})(z - \bar{z}_{j}).
	\end{equation*}
	Since $ \bar{z}_{j} $ is the projection, $ (z - \bar{z}_{j}) $ is the normal vector at $ \bar{z}_{j} $, which is aligned with the gradient $ \tilde{\triangledown}_{y}q_{j}(\bar{z}_{j}) $. Hence $ \tilde{\triangledown}_{y}q_{j}(\bar{z}_{j})(z - \bar{z}_{j}) \geq 0 $, i.e. $ l_{j}(z,z) \ge 0, \forall j = 1,2,..., M $, i.e. $ z \in F_{z} $.
	
	Furthermore, since $ q_{j}(y) $ is a convex function, we have for any $ y \in F_{z} $,
	\begin{equation*}
	q_{j}(y) \geq q_{j}(\bar{z}_{j}) + \tilde{\triangledown}_{y}q_{j}(\bar{z}_{j})(y - \bar{z}_{j}) 
	= l_{j}(y,z) \geq 0,
	\end{equation*}
	which means $ y \in F $. Hence $ F_{z} \subseteq F $.
\end{proof}

\subsection{The \texttt{SCvx-fast} Algorithm}
Now that we have a convex and compact feasible region $ F_{z} $ and a convex objective function $ P(y) $, we are ready to present a successive procedure to solve the non-convex problem in~\eqref{eq:final_prob}. Note that the feasible region $ F_{z} $ is defined by $ z $. Therefore, if we start from a point $ z^{(0)} \in F $, a sequence $ \{z^{(k)} \} $ will be generated, where
\begin{equation} \label{eq:sub_prob}
z^{(k+1)} = \underset{y}{argmin}\{P(y)\ |\ y\in F_{z^{(k)}} \}, \quad k=0,1,...
\end{equation}
This is a convex programming subproblem, whose global minimizer is attained at $ z^{(k+1)} $. At these intermediate steps, $ z^{(k+1)} $ may not be the optimal solution to~\eqref{eq:final_prob}. Our goal, however, is to prove that this sequence $ \{z^{(k)} \} $ converges to a limit point $ z^{*} $, and that this limit point solves~\eqref{eq:final_prob} by \emph{project-and-convexify} at $ z^{*} $ itself, i.e., it is a ``fixed-point" satisfying
\begin{equation} \label{eq:fast_optimal}
z^{*} = \underset{y}{argmin}\{P(y)\ |\ y\in F_{z^{*}} \}.
\end{equation}
More importantly, we want to show that $ z^{*} $ gives a local optimum to~\eqref{eq:final_prob} convexified at $ z^{*} $ itself. Then by solving a sequence of convex programming subproblems, we effectively solved the non-convex optimal control problem in~\eqref{eq:ori_prob} thanks to Theorem~\ref{thm:fast_exactness}. Notably, we do not use trust regions in this process, hence no trust-region updating mechanism involved. As a result, this procedure converges much faster than \scvx, and thus we call it the \texttt{SCvx-fast} algorithm. It is summarized in Algorithm~\ref{algo:SCvx-fast}.

\begin{algorithm}[!h]
	\caption{The \texttt{SCvx-fast} Algorithm}
	\label{algo:SCvx-fast}
	\begin{algorithmic}[1]
		\Procedure{\fast}{$z^{(0)},\lambda$}
		\State \textbf{input} Initial point $ z^{(0)} $. Penalty weight $ \lambda \geq 0 $.
		\If{$ z^{(0)} \notin F $}
			\State \textbf{Infeasibility Preprocessing} Execute Algorithm~\ref{algo:infeas-ini} to get a new $ \tilde{z}^{(0)} \in F $.
		\EndIf
		\While {not converged}
		\State \textbf{step 1 Project} At each succession $ k $, we have the current point $ z^{(k)} $. For each constraint $ q_{j}(y) $, solve the problem in~\eqref{eq:project} to get a projection point $ \bar{z_{j}^{(k)}} $.
		\State \textbf{step 2 Convexify} Construct the convex feasible region $ F_{z^{(k)}} $ by using~\eqref{eq:linearization}.
		\State \textbf{step 3} Solve the convex subproblem in~\eqref{eq:sub_prob} to get $ z^{(k+1)} $. Let $ z^{(k)}\leftarrow z^{(k+1)} $ and go to the next succession.
		\EndWhile
		\State \textbf{return} $ z^{(k+1)} $.
		\EndProcedure
	\end{algorithmic}
\end{algorithm}

\section{Convergence Analysis} \label{sec:conv}

\subsection{Global Convergence}
In this section, we proceed to show that Algorithm~\ref{algo:SCvx-fast} does converge to a point $ z^{*} $ that indeed satisfies~\eqref{eq:fast_optimal}. First we must assume the application of regular constraint qualifications, namely the Linear Independence Constraint Qualification (LICQ) and the Slater's condition. They can be formalized as the following:

\begin{assumption}[\textbf{LICQ}] \label{asup:fast_licq}
	For each $ z \in F $, the generalized Jacobian matrix $ \tilde{\triangledown}_{y} q(z) $ has full row rank, i.e. rank($ \tilde{\triangledown}_{y} q(z) $) $ = M $.
\end{assumption}

\begin{assumption}[\textbf{Slater's condition}] \label{asup:slater}
	For each $ z \in F $, the convex feasible region $ F_{z} $ contains interior points.
\end{assumption}

These assumptions do impose some practical restrictions on our feasible region. Figure~\ref{fig:licq} shows some examples where LICQ or Slater's condition might fail. For scenarios like (a), we may perturb our discrete points to break symmetry, For scenarios like (b), we have to assume the connectivity of the feasible region. In other words, our feasible region cannot be degenerate at some point, for example, in collision avoidance the feasible region is not allowed to be completely obstructed.
\begin{figure}[!h]
	\begin{center}
		\includegraphics[width=0.8\textwidth]{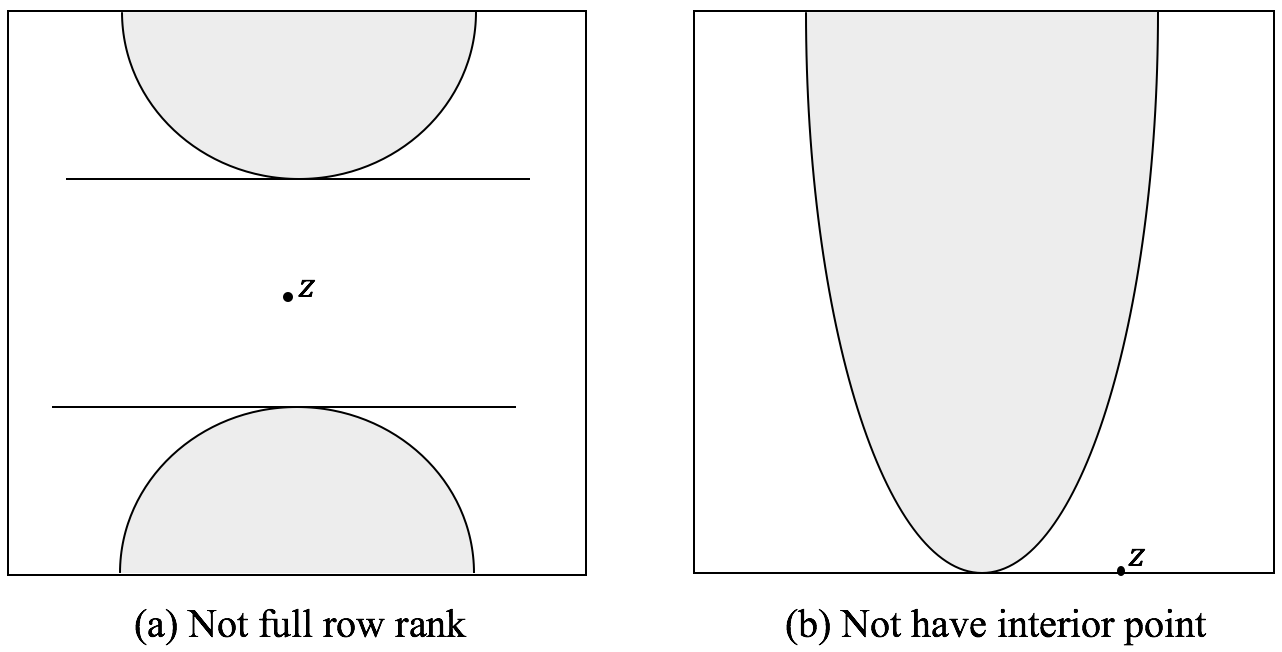}
		\caption{Cases where LICQ or Slater's condition fails} 
		\label{fig:licq}
	\end{center}
\end{figure}

To analyze convergence, first we need to show that the point-to-set mapping $ F_{z} $ is continuous in the sense that given any point $ z^{(1)} \in F $ and $ y^{(1)} \in F_{z^{(1)}} $, then for any point $ z^{(2)} \in F $ in the neighborhood of $ z^{(1)} $, there exists a $ y^{(2)} \in F_{z^{(2)}} $ that is close enough to $ y^{(1)} $.

First, We have the following lemma:
\begin{lemma}[\textbf{Lipschitz continuity of $ l(y,z) $}] \label{lem:fast_lipschitz}
	The linear approximation $ l(y,z) $ in~\eqref{eq:linearization} is Lipschitz continuous in $ z $, that is
	\begin{equation} \label{eq:fast_lipschitz}
	\| l(y,z^{(1)}) - l(y,z^{(2)}) \| \leq \gamma \| z^{(1)} - z^{(2)} \|,
	\end{equation}
	for any $ z^{(1)}, z^{(2)} \in F $ and constant $ \gamma $.
\end{lemma}
\begin{proof}
	Each $ l_{j}(y,z) $ is in fact a composition of two mappings. The first mapping maps $ z $ to its projection $ \bar{z}_{j} $. This mapping is defined by the optimization problem in~\eqref{eq:project}. It is a well-known result that this mapping is non-expansive (i.e. Lipschitz continuous with constant 1), See e.g.~\cite{wulbert_project}. The second mapping is defined by the auxiliary function:
	\begin{equation*}
	a_{j}(y,\bar{z}_{j}) = \tilde{\triangledown}_{y}q_{j}(\bar{z}_{j})(y - \bar{z}_{j}),
	\end{equation*}
	If $ q(y) \in C^{2}(Y) $, then $ \tilde{\triangledown}_{y}q_{j}(\bar{z}_{j}) $ is obviously Lipschitz continuous. When $ q(y) $ is non-smooth, since we are taking the gradient to the supporting hyerplane at $ \bar{z}_{j} $ as $ \tilde{\triangledown}_{y}q_{j}(\bar{z}_{j}) $, it is also naturally Lipschitz continuous sicne it is constant. Therefore, $ a_{j}(y,\bar{z}_{j}) $ is always Lipschitz continuous in $ \bar{z}_{j} $.
	By the composition rule of two Lipschitz continuous functions, we have $ l_{j}(y,z) $ is Lipschitz continuous, for all $ j = 1,2,..., M $. Therefore, sum over $ j $ gives the Lipschitz continuity of $ l(y,z) $.
\end{proof}

Now with Assumptions~\ref{asup:fast_licq},~\ref{asup:slater} and Lemma~\ref{lem:fast_lipschitz}, we are ready to prove the continuity of point-to-set mapping $ F_{z} $. The result is given as follows:

\begin{lemma}[\textbf{Continuity of mapping $ F_{z} $}] \label{lem:fast_continuity}
	Given $ z^{(1)} \in F $ and $ y^{(1)} \in F_{z^{(1)}} $, then given $ \epsilon > 0 $, there exists a $ \delta > 0 $ so that for any point $ z^{(2)} \in F $ with $ \|z^{(2)} - z^{(1)}\| \leq \delta $, there exists a $ y^{(2)} \in F_{z^{(2)}} $ such that $ \|y^{(2)} - y^{(1)}\| \leq \epsilon $.
\end{lemma}
\begin{proof}
	From Assumption~\ref{asup:fast_licq}, we know that the generalized Jacobian matrix $ \tilde{\triangledown}_{y} q(z) $ has full row rank. Thus matrix $ \tilde{\triangledown}_{y} q(z)\tilde{\triangledown}_{y} q(z)^{T} $ is symmetric and positive definite for any $ z \in F $. Consequently, there exists $ \beta $ such that
	\begin{equation} \label{eq:deriv_ineq}
	\| (\tilde{\triangledown}_{y} q(z)\tilde{\triangledown}_{y} q(z)^{T})^{-1} \| \leq \beta^{2}, \quad \forall z \in F.
	\end{equation}
	If $ y^{(1)} \in F_{z^{(2)}} $, then take $ y^{(2)} = y^{(1)} $ such that $ \|y^{(2)} - y^{(1)}\| =0 \leq \epsilon $. Now suppose $ y^{(1)} \notin F_{z^{(2)}} $, then there exists at least one $ j $, such that $ l_{j}(y^{(1)}, z^{(2)}) < 0 $. Let
	\begin{equation*}
	\bar{l}_{j} = \begin{cases}
	l_{j}(y^{(1)}, z^{(2)}) & \quad l_{j}(y^{(1)}, z^{(2)}) < 0, \\
	0 & \quad l_{j}(y^{(1)}, z^{(2)}) \geq 0.
	\end{cases}
	\end{equation*}
	Note that $ l_{j}(y^{(1)}, z^{(1)}) \geq 0 $, then by definition,
	\begin{equation*}
	| \bar{l}_{j} | \leq | l_{j}(y^{(1)}, z^{(1)}) - l_{j}(y^{(1)}, z^{(2)}) |, \quad \forall j = 1, 2, ..., M.
	\end{equation*}
	Hence we have
	\begin{equation} \label{eq:norm_L_bar}
	\| \bar{l} \| \leq \| l(y^{(1)}, z^{(1)}) - l(y^{(1)}, z^{(2)}) \| \leq \gamma \| z^{(1)} - z^{(2)} \|.
	\end{equation}
	The second inequality follows from~\eqref{eq:fast_lipschitz} in Lemma~\ref{lem:fast_lipschitz}.
	
	Now we consider two cases: \\
	\textbf{Case 1}: $ y^{(1)} $ is an interior point of $ Y $, then $ \exists \, \epsilon_{1} \in (0, \epsilon] $, such that $ y \in Y $ for all $ y $ satisfies $ \| y - y^{(1)} \| \leq \epsilon_{1} $. \\
	Let $ \|z^{(2)} - z^{(1)}\| \leq \delta $, with $ \delta = \epsilon_{1}/\beta\gamma $, and let
	\begin{equation} \label{eq:y2}
	y^{(2)} = y^{(1)} - \tilde{\triangledown}_{y} q(z^{(2)})^{T} (\tilde{\triangledown}_{y} q(z^{(2)})\tilde{\triangledown}_{y} q(z^{(2)})^{T})^{-1} \bar{l}.
	\end{equation}
	We need to verify that $ y^{(2)} \in F_{z^{(2)}}$. First, we have
	\begin{equation} \label{eq:Lj2}
	l_{j}(y^{(2)},z^{(2)}) = \tilde{\triangledown}_{y}q_{j}(\bar{z}_{j}^{(2)})(y^{(2)} - \bar{z}_{j}^{(2)}).
	\end{equation}
	Substitute $ y^{(2)} $ from~\eqref{eq:y2} into the stacked form of~\eqref{eq:Lj2}, and then unstack, we get
	\begin{align*}
	l_{j}(y^{(2)},z^{(2)}) &= \tilde{\triangledown}_{y}q_{j}(\bar{z}_{j}^{(2)})(y^{(1)} - \bar{z}_{j}^{(2)}) - \bar{l}_{j} \\
	&= l_{j}(y^{(1)},z^{(2)}) - \bar{l}_{j} \\
	&\geq 0.
	\end{align*}
	The last inequality follows from the definition of $ \bar{l}_{j} $. Therefore, $ l(y^{(2)},z^{(2)}) \geq 0 $. \\
	Next we need to show that $ y^{(2)} \in Y $. Rearrange terms in~\eqref{eq:y2}, we have
	\begin{align*}
	\| y^{(2)} - y^{(1)} \|^{2} &= \bar{l}^{T} (\tilde{\triangledown}_{y} q(z^{(2)})\tilde{\triangledown}_{y} q(z^{(2)})^{T})^{-1} \bar{l} \\
	&\leq \beta^{2} \| \bar{l} \|^{2} \\
	&\leq \beta^{2} \gamma^{2} \| z^{(1)} - z^{(2)} \|^{2}
	\end{align*}
	The inequalities follow from~\eqref{eq:deriv_ineq} and~\eqref{eq:norm_L_bar}. Thus we have $ \| y^{(2)} - y^{(1)} \| \leq \beta\gamma\delta = \epsilon_{1} $, i.e. $ y^{(2)} \in Y $. So now we have verified $ y^{(2)} \in F_{z^{(2)}}$. Since $ \epsilon_{1} \leq \epsilon $, we have $ \| y^{(2)} - y^{(1)} \| \leq \epsilon $, so that $ F_{z} $ is continuous.
	
	\textbf{Case 2}: $ y^{(1)} $ is a boundary point of $ Y $. From Assumption~\ref{asup:slater}, $ F_{z^{(1)}} $ has interior points. Also since $ F_{z^{(1)}} $ is a convex set, it is a well known fact that there are interior points in the $ \epsilon- $Neighborhood of every point in $ F_{z^{(1)}} $. Then we can apply the same argument as in Case 1, to get that $ F_{z} $ is continuous.
\end{proof}

One way to show convergence is to demonstrate the convergence of objective functions, $ P(z^{(k)}) $. Now let's define
\begin{equation*}
\Phi (z) := \underset{y}{min}\{P(y)\ |\ y\in F_{z} \}
\end{equation*}
to be the function that maps the point $ z $ we are solving at in each iteration to the optimal value of the objective function in $ F_{z} $. By using the continuity of the point-to-set mapping $ F_{z} $, the following lemma gives the continuity of the function $ \Phi (z) $.

\begin{lemma} \label{lem:continuity_phi}
	$ \Phi (z) $ is continuous for $ z \in F $.
\end{lemma}
\begin{proof}
	Given any two points $ z^{(1)}, z^{(2)} \in F $, and they are close to each other, i.e. $ \| z^{(1)} - z^{(2)} \| \leq \delta $. Let
	\begin{align*}
	& y^{(1)} = \underset{y}{argmin}\{P(y)\ |\ y\in F_{z^{(1)}} \} \; \text{and} \\
	& y^{(2)} = \underset{y}{argmin}\{P(y)\ |\ y\in F_{z^{(2)}} \},
	\end{align*}
	then we have $ \Phi (z^{(1)})=P(y^{(1)}) $ and $ \Phi (z^{(2)})=P(y^{(2)}) $. \\
	Without loss of generality, let $ \Phi (z^{(1)}) \leq \Phi (z^{(2)}) $, i.e. $ P(y^{(1)}) \leq P(y^{(2)}) $. From Lemma~\ref{lem:fast_continuity}, the continuity of $ F_{z} $, there exists $ \hat{y}^{(2)} \in F_{z^{(2)}} $, such that $ \| \hat{y}^{(2)} - y^{(1)} \| \leq \eta $ for any $ \eta >0 $. Then since $ P(y) $ is continuous, we have
	\begin{equation} \label{eq:P_conti}
	| P(\hat{y}^{(2)}) - P(y^{(1)}) | \leq \epsilon.
	\end{equation}
	Since $ y^{(2)} $ is the minimizer of $ P(y) $ in $ F_{z^{(2)}} $, $ P(y^{(2)}) \leq P(\hat{y}^{(2)}) $. Therefore, $ P(y^{(1)}) \leq P(\hat{y}^{(2)}) $ by assumption. Now~\eqref{eq:P_conti} becomes $ P(\hat{y}^{(2)}) - P(y^{(1)}) \leq \epsilon $. Again, because $ P(y^{(2)}) \leq P(\hat{y}^{(2)}) $, we have $ P(y^{(2)}) - P(y^{(1)}) \leq \epsilon $, i.e. $ \Phi (z^{(2)}) - \Phi (z^{(1)}) \leq \epsilon $, which means $ \Phi (z) $ is continuous.
\end{proof}

With the continuity of $ \Phi(z) $, we are finally ready to present the final convergence results:
\begin{theorem}[\textbf{Global convergence}]
	Under Assumptions~\ref{asup:convex_f},~\ref{asup:convex_h},~\ref{asup:fast_licq}, and~\ref{asup:slater}, the sequence $ \{z^{(k)} \} $ generated by the successive procedure~\eqref{eq:sub_prob} is in $ F $, and has limit point $ z^{*} $, at which the corresponding sequence $ \{P(z^{(k)})\} $ attains its minimum, $ P(z^{*}) $. \\
	More importantly, $ P(z^{*}) = \Phi(z^{*}) $, i.e. $ z^{*} $ is a local optimum of the penalty problem in~\eqref{eq:final_prob} convexified at $ z^{*} $.
\end{theorem}
\begin{proof}
	From Lemma~\ref{lem:invariance}, we have $ z^{(k)} \in F_{z^{(k)}} \subseteq F $, then
	\begin{equation*}
	P(z^{(k+1)}) = \underset{y}{min}\{P(y)\ |\ y\in F_{z^{(k)}} \} \leq P(z^{(k)}),
	\end{equation*}
	because $ z^{(k)} $ is a feasible point to this convex optimization problem, while $ z^{(k+1)} $ is the optimum. Therefore, the sequence $ \{P(z^{(k)})\} $ is monotonically decreasing. \\
	Also since $ F_{z^{(k)}} \subseteq F $ for all $ k $, we have
	\begin{equation*}
	P(z^{(k+1)}) \geq \underset{y}{min}\{P(y)\ |\ y\in F \},
	\end{equation*}
	which means the sequence $ \{P(z^{(k)})\} $ is bounded from below. Then by the monotone convergence theorem, see e.g. \cite{rudin1964principles}, $ \{P(z^{(k)})\} $ converges to its infimum. Due to the compactness of $ F_{z^{(k)}} $, this infimum is attained by all the convergent subsequences of $ \{z^{(k)}\} $. Let $ z^{*} \in F $ be one of the limit points, then $ \{P(z^{(k)})\} $ attains its minimum at $ P(z^{*}) $, i.e.
	\begin{equation} \label{eq:infimum}
	P(z^{*}) \leq P(z^{(k)}), \quad \forall \; k = 0, 1,...
	\end{equation}
	
	To show $ P(z^{*}) = \Phi(z^{*}) $, we note that since $ z^{*} \in F $, $ z^{*} \in F_{z^{*}} $ by Lemma~\ref{lem:invariance}. Thus we have
	\begin{equation*}
	\Phi(z^{*}) := \underset{y}{min}\{P(y)\ |\ y\in F_{z^{*}} \} \leq P(z^{*}).
	\end{equation*}
	Now for the sake of contradiction, we suppose $ \Phi(z^{*}) < P(z^{*}) $. Then by Lemma~\ref{lem:continuity_phi}, the continuity of $ \Phi(z) $, there exists a sufficiently large $ k $ such that $ \Phi(z^{k}) < P(z^{*}) $. Then we have $ P(z^{(k+1)}) = \Phi(z^{k}) < P(z^{*}) $, which contradicts~\eqref{eq:infimum}. Therefore $ P(z^{*}) = \Phi(z^{*}) $ holds.
\end{proof}

\subsection{Superlinear Convergence}
We first denote the inverse mapping of $ F_z $ as $ \bar{F_z}: F_z \rightarrow z $, which is a set-to-point mapping. Similar to Lemma~\ref{lem:fast_continuity}, we will first show the continuity of $ \bar{F_z} $.

\begin{lemma}[\textbf{Lipschitz continuity of $ l(y,z) $}] \label{lem:new_fast_lipschitz}
	The linear approximation $ l(y,z) $ in~\eqref{eq:linearization} is Lipschitz continuous in $ z $, that is
	\begin{equation} \label{eq:new_fast_lipschitz}
	\| l(y,z^{(1)}) - l(y,z^{(2)}) \| \geq \bar{\gamma} \| z^{(1)} - z^{(2)} \|,
	\end{equation}
	for any $ z^{(1)}, z^{(2)} \in F $ and constant $ \bar{\gamma} $.
\end{lemma}
\begin{proof}
	Similar to Lemma~\ref{lem:fast_lipschitz}.
\end{proof}

\begin{lemma}[\textbf{Continuity of mapping $ \bar{F_z} $}] \label{lem:inverse_continuity}
	Given $ z^{(1)}, z^{(2)} \in F $ and $ \epsilon > 0, \delta > 0 $, if for any $ y^{(1)} \in F_{z^{(1)}} $, there exists $ y^{(2)} \in F_{z^{(2)}} $ such that $ \|y^{(2)} - y^{(1)}\| \leq \epsilon $, then $ \|z^{(2)} - z^{(1)}\| \leq \delta $.
\end{lemma}
\begin{proof}
	If $ y^{(1)} \in F_{z^{(2)}} $, then take $ y^{(2)} = y^{(1)} $ such that $ \|y^{(2)} - y^{(1)}\| =0 \leq \epsilon $. Now suppose $ y^{(1)} \notin F_{z^{(2)}} $, then there exists at least one $ j $, such that $ l_{j}(y^{(1)}, z^{(2)}) < 0 $. Let
	\begin{equation*}
	\bar{l}_{j} = \begin{cases}
	l_{j}(y^{(1)}, z^{(2)}) & \quad l_{j}(y^{(1)}, z^{(2)}) < 0, \\
	0 & \quad l_{j}(y^{(1)}, z^{(2)}) \geq 0.
	\end{cases}
	\end{equation*}
	Note that $ l_{j}(y^{(1)}, z^{(1)}) \geq 0 $, then by construction,
	\begin{equation*}
	| \bar{l}_{j} | \geq | l_{j}(y^{(1)}, z^{(1)}) - l_{j}(y^{(1)}, z^{(2)}) |, \quad \forall j = 1, 2, ..., M.
	\end{equation*}
	Hence we have
	\begin{equation} \label{eq:new_norm_L_bar}
	\| \bar{l} \| \geq \| l(y^{(1)}, z^{(1)}) - l(y^{(1)}, z^{(2)}) \| \geq \bar{\gamma} \| z^{(1)} - z^{(2)} \|.
	\end{equation}
	The second inequality follows from~\eqref{eq:new_fast_lipschitz} in Lemma~\ref{lem:new_fast_lipschitz}.
	
	Now we consider two cases: \\
	\textbf{Case 1}: $ y^{(1)} $ is an interior point of $ Y $, then $ \exists \, \epsilon_{1} \in (0, \epsilon] $, such that $ y \in Y $ for all $ y $ satisfies $ \| y - y^{(1)} \| \leq \epsilon_{1} $. \\
	Let $ \delta = \epsilon_{1}/\beta\gamma $, and let
	\begin{equation} \label{eq:new_y2}
	y^{(2)} = y^{(1)} - \tilde{\triangledown}_{y} q(z^{(2)})^{T} (\tilde{\triangledown}_{y} q(z^{(2)})\tilde{\triangledown}_{y} q(z^{(2)})^{T})^{-1} \bar{l}.
	\end{equation}
	We need to verify that $ y^{(2)} \in F_{z^{(2)}}$. First, we have
	\begin{equation} \label{eq:new_Lj2}
	l_{j}(y^{(2)},z^{(2)}) = \tilde{\triangledown}_{y}q_{j}(\bar{z}_{j}^{(2)})(y^{(2)} - \bar{z}_{j}^{(2)}).
	\end{equation}
	Substitute $ y^{(2)} $ from~\eqref{eq:new_y2} into the stacked form of~\eqref{eq:new_Lj2}, and then unstack, we get
	\begin{align*}
	l_{j}(y^{(2)},z^{(2)}) &= \tilde{\triangledown}_{y}q_{j}(\bar{z}_{j}^{(2)})(y^{(1)} - \bar{z}_{j}^{(2)}) - \bar{l}_{j} \\
	&= l_{j}(y^{(1)},z^{(2)}) - \bar{l}_{j} \\
	&\geq 0.
	\end{align*}
	The last inequality follows from the definition of $ \bar{l}_{j} $. Therefore, $ l(y^{(2)},z^{(2)}) \geq 0 $. \\
	Next we need to show that $ y^{(2)} \in Y $. Rearrange terms in~\eqref{eq:new_y2}, we have
	\begin{align*}
	\| y^{(2)} - y^{(1)} \|^{2} &= \bar{l}^{T} (\tilde{\triangledown}_{y} q(z^{(2)})\tilde{\triangledown}_{y} q(z^{(2)})^{T})^{-1} \bar{l} \\
	&\geq \beta^{2} \| \bar{l} \|^{2} \\
	&\geq \beta^{2} \gamma^{2} \| z^{(1)} - z^{(2)} \|^{2}
	\end{align*}
	The inequalities follow from~\eqref{eq:deriv_ineq} and~\eqref{eq:new_norm_L_bar}. Thus we have $ \| y^{(2)} - y^{(1)} \| \leq \beta\gamma\delta = \epsilon_{1} $, i.e. $ y^{(2)} \in Y $. So now we have verified $ y^{(2)} \in F_{z^{(2)}}$. Since $ \epsilon_{1} \leq \epsilon $, we have $ \|z^{(2)} - z^{(1)}\| \leq \delta $, so that $ \bar{F_z} $ is continuous.
	
	\textbf{Case 2}: $ y^{(1)} $ is a boundary point of $ Y $. From Assumption~\ref{asup:slater}, $ F_{z^{(1)}} $ has interior points. Also since $ F_{z^{(1)}} $ is a convex set, it is a well known fact that there are interior points in the $ \epsilon- $Neighborhood of every point in $ F_{z^{(1)}} $. Then we can apply the same argument as in Case 1, to get that $ \bar{F_{z}} $ is continuous.
\end{proof}

By using the continuity of the set-to-point mapping $ \bar{F_{z}} $, we have the following lemma:
\begin{lemma} \label{lem:growth}
	Given $ z^{(1)}, z^{(2)} \in F $ and $ \gamma > 0 $, we have
	\begin{equation} \label{eq:growth}
	P(y^{(2)}) - P(y^{(1)}) \geq \gamma \| z^{(1)} - z^{(2)} \|.
	\end{equation}
\end{lemma}
\begin{proof}
	Given any two points $ z^{(1)}, z^{(2)} \in F $, and they are close to each other, i.e. $ \| z^{(1)} - z^{(2)} \| \leq \delta $. Let
	\begin{align*}
	& y^{(1)} = \underset{y}{argmin}\{P(y)\ |\ y\in F_{z^{(1)}} \} \; \text{and} \\
	& y^{(2)} = \underset{y}{argmin}\{P(y)\ |\ y\in F_{z^{(2)}} \},
	\end{align*}
	then we have $ \Phi (z^{(1)})=P(y^{(1)}) $ and $ \Phi (z^{(2)})=P(y^{(2)}) $. \\
	Without loss of generality, let $ \Phi (z^{(1)}) \leq \Phi (z^{(2)}) $, i.e. $ P(y^{(1)}) \leq P(y^{(2)}) $. From Lemma~\ref{lem:inverse_continuity}, the continuity of $ \bar{F_{z}} $, there exists $ \hat{y}^{(2)} \in F_{z^{(2)}} $, such that $ \| \hat{y}^{(2)} - y^{(1)} \| \leq \eta $ for any $ \eta >0 $. Then since $ P(y) $ is continuous, we have
	\begin{equation} \label{eq:new_P_conti}
	| P(\hat{y}^{(2)}) - P(y^{(1)}) | \leq \epsilon.
	\end{equation}
	Since $ y^{(2)} $ is the minimizer of $ P(y) $ in $ F_{z^{(2)}} $, $ P(y^{(2)}) \leq P(\hat{y}^{(2)}) $. Therefore, $ P(y^{(1)}) \leq P(\hat{y}^{(2)}) $ by assumption. Now~\eqref{eq:P_conti} becomes $ P(\hat{y}^{(2)}) - P(y^{(1)}) \leq \epsilon $. Again, because $ P(y^{(2)}) \leq P(\hat{y}^{(2)}) $, we have $ P(y^{(2)}) - P(y^{(1)}) \leq \epsilon $, i.e. $ \Phi (z^{(2)}) - \Phi (z^{(1)}) \leq \epsilon $, which means $ \Phi (z) $ is continuous.
\end{proof}

Lemma~\ref{lem:growth} provides an important condition that lower bounds the cost reduction rate. Next we show that given this condition, the \fast procedure will indeed converge superlinearly. To proceed, first denote the stack of $ J(\cdot) $ and $ g(\cdot) $ as $ G(\cdot) $, and represent $ P(\cdot) $ by a function composition $ \psi(G(\cdot)) $. Qualitatively we can write
$
\psi(G(\cdot)) = G_{cost}(\cdot) + \|G_{eq}(\cdot)\|_1.
$
Note that $ \psi(\cdot) $ is convex since $ \|\cdot\|_1 $ is a convex function.

\begin{lemma}
	Following Lemma~\ref{lem:growth}, there exists $ \gamma >0 $, such that
	\begin{equation} \label{eq:ineq_critical}
	\psi\left( G(\bar{z}) + \nabla G(\bar{z})^T \, d \right) \geq \psi(G(\bar{z})) + \gamma \| d \|, \quad \forall \, d \in F_{\bar{z}}.
	\end{equation}
\end{lemma}
\begin{proof}
	First we show that the statement is true for any small step $ d_\gamma $ such that $ \|d_\gamma\| \leq \gamma $, where $ \gamma $ is defined in Lemma~\ref{lem:growth}. We have
	$\psi\left( G(\bar{z}) + \nabla G(\bar{z})^T \, d_\gamma \right) - \psi(G(\bar{z})) = \psi(G(\bar{z}+d_\gamma)) - \psi(G(\bar{z})) 
	+ \psi\left( G(\bar{z}) + \nabla G(\bar{z})^T \, d_\gamma \right) - \psi(G(\bar{z}+d_\gamma)) 
	\geq \beta \|d_\gamma\| + \orderof(\|d_\gamma\|) \geq \;\frac{\beta}{2} \|d_\gamma\|.$
	Let $ \gamma := \frac{\beta}{2} $, then we have \eqref{eq:ineq_critical} hold for a small step $ d_\gamma $.
	Note that the first inequality is due to Lemma~\ref{lem:growth} and again the fact that $ \orderof(\|d_\gamma\|) $ can be taken out of $ |\cdot| $ and $ \max(0,\,\cdot\,) $ and $ G(\cdot) \in C^1 $.
	
	Now to generalize this to any $ d \in F_{\bar{z}}, d \neq 0 $, we first define
	$ \zeta := \min\left( 1, \gamma/\|d\| \right), $
	and let $z_\zeta := \bar{z} + \zeta d $. Denoting $ (z_\zeta - \bar{z}) $ as $ d_\zeta $, we have
	$ d_\zeta = \zeta \, d $.
	With this definition of $ \zeta $, one can verify that $ \|d_\zeta\|\leq \gamma $. Hence, \eqref{eq:ineq_critical} holds true for $ d_\zeta $. Therefore, we have
	\begin{align*}
	\gamma \zeta \|d\| = \gamma \|d_\zeta\| &\leq \psi\left( G(\bar{z}) + \nabla G(\bar{z})^T \, d_\zeta \right) - \psi(G(\bar{z})) \\
	&= \psi\left( G(\bar{z}) + \zeta \nabla G(\bar{z})^T \, d \right) - \psi(G(\bar{z})) \\
	&= \psi\left( (1-\zeta)G(\bar{z}) + \zeta\left[ G(\bar{z}) + \nabla G(\bar{z})^T \, d \right] \right) - \psi(G(\bar{z})) \\
	&\leq (1-\zeta)\psi(G(\bar{z})) + \zeta \psi\left(G(\bar{z}) + \nabla G(\bar{z})^T \, d \right) - \psi(G(\bar{z})) \\
	&=\zeta \left[ \psi\left(G(\bar{z}) + \nabla G(\bar{z})^T \, d \right) - \psi(G(\bar{z})) \right].
	\end{align*}
	The second inequality is due to the convexity of $ \psi $ (note that $ 0 \leq \zeta \leq 1 $). Dividing both sides by $ \zeta $, we obtain \eqref{eq:ineq_critical}.
\end{proof}

\begin{theorem} [\textbf{Superlinear Convergence}] \label{thm:superlinear}
	Given a sequence $ \{z^k\} \rightarrow \bar{z} $ generated by the \fast algorithm (Algorithm~\ref{algo:SCvx-fast}) and suppose that Assumption~\ref{asup:fast_licq}, \ref{asup:slater}, \ref{asup:convex_f}, and \ref{asup:convex_h} are satisfied, then we have
	$
	\|y^{k+1} - \bar{y}\| = \orderof(\|y^{k} - \bar{y}\|).
	$
\end{theorem}
\begin{proof}
	Since $ d^k := y^{k+1} - y^k $ is the optimal solution to the convex subproblem within $ F_{\bar{z}} $, we have $ \forall \, d \in F_{\bar{z}} $,
	$ \psi\big( G(y^k) + \nabla G(y^k)^T\,d \big) \geq \psi\big( G(y^k) + \nabla G(y^k)^T\,d^k \big). $
	Let $ d = \bar{y} - y^k $, we then have
	\begin{equation*}
	\psi\left( G(y^k) + \nabla G(y^k)^T\,(\bar{y} - y^k) \right) - \psi\left( G(y^k) + \nabla G(y^k)^T\,(y^{k+1} - y^k) \right) \geq 0.
	\end{equation*}
	Together with~\eqref{eq:ineq_critical}, we get
	\begin{align*}
	&\gamma \|y^{k+1} - \bar{y}\| \leq \;\psi\left( G(\bar{y}) + \nabla G(\bar{y})^T\,(y^{k+1} - \bar{y})  \right) - \psi\big(G(\bar{y})\big) \\
	&+ \psi\left( G(y^k) + \nabla G(y^k)^T\,(\bar{y} - y^k) \right) - \psi\left( G(y^k) + \nabla G(y^k)^T\,(y^{k+1} - y^k) \right).
	\end{align*}
	Since $ \psi $ is convex and has a compact domain, it is (locally) Lipschitz continuous (see e.g. \cite{convex_lipschitz}). Therefore, we have $\psi\left( G(y^k) + \nabla G(y^k)^T\,(\bar{y} - y^k) \right) - \psi\big(G(\bar{y})\big) \leq L \|G(y^k) - G(\bar{y}) + \nabla G(y^k)^T\,(\bar{y} - y^k) \|,$
	where $ L $ is the Lipschitz constant, and 
	\begin{align*}
	&\psi\left( G(\bar{y}) + \nabla G(\bar{y})^T\,(y^{k+1} - \bar{y})  \right) - \psi\left( G(y^k) + \nabla G(y^k)^T\,(y^{k+1} - y^k) \right) \\
	&\quad \leq L \|G(\bar{y}) - G(y^k) - \nabla G(y^k)^T\,(\bar{y} - y^k) + \left[ \nabla G(\bar{y})-\nabla G(y^k)\right]^T \left( y^{k+1} - \bar{y}\right) \| \\
	&\quad\leq L \|G(\bar{y}) - G(y^k) - \nabla G(y^k)^T\,(\bar{y} - y^k) \| + L \|\left[ \nabla G(\bar{y})-\nabla G(y^k) \right]^T \left( y^{k+1} - \bar{y}\right) \|.
	\end{align*}
	Combining the two parts, we obtain 
	$
	\|y^{k+1} - \bar{y}\| \leq \frac{2L}{\gamma} \|G(\bar{y}) - G(y^k) - \nabla G(y^k)^T\,(\bar{y} - y^k) \| + \frac{L}{\gamma} \|\left[ \nabla G(\bar{y})-\nabla G(y^k) \right]^T \left( y^{k+1} - \bar{y}\right) \|.
	$
	Since $ G(\cdot) \in C^1 $, we have the first term
	$ \|G(\bar{y}) - G(y^k) - \nabla G(y^k)^T\,(\bar{y} - y^k) \| = \orderof(\|y^{k} - \bar{y}\|). $
	Given the fact that as $ k\rightarrow \infty$, $(\nabla G(\bar{y})-\nabla G(y^k))\rightarrow 0 $, we also have the second term
	$$ \|\left[ \nabla G(\bar{y})-\nabla G(y^k) \right]^T \left( y^{k+1} - \bar{y}\right) \| = \orderof(\|y^{k+1} - \bar{y}\|). $$
	Therefore, we have
	$ \|y^{k+1} - \bar{y}\| = \orderof(\|y^{k} - \bar{y}\|). $
\end{proof}

\section{Numercial Results} \label{sec:num}
In this section, we apply the \fast algorithm to an aerospace problem and present the numerical results. Consider a multi--rotor vehicle with state at time $t_i$ given by $x_i = \left\lbrack p_i^T,v_i^T \right\rbrack^T$, where $p_i\in\mathbb{R}^3$ and $v_i\in\mathbb{R}^3$ represent vehicle position and velocity at time $t_i$ respectively. We assume that vehicle motion is adequately modeled by double integrator dynamics with constant time step $\Delta t$ such that $x_{i+1} = Ax_i + B(u_i+g)$, where $u_i\in\mathbb{R}^3$ is the control at time $t_i$, $g\in\mathbb{R}^3$ is a constant gravity vector, $A$ is the discrete state transition matrix, and $B$ utilizes zero order hold integration of the control input.
Further, we impose a speed upper--bound at each time step, $\|v_i\|_2\leq V_{\max}$, an acceleration upper--bound such that $\|u_i\|_2\leq u_{\max}$ (driven by a thrust upper--bound), and a thrust cone constraint $\hat{n}^Tu_i \geq \|u_i\|_2\cos(\theta_{cone})$ that constrains the thrust vector to a cone pointing towards the unit vector $\hat{n}$ (pointing towards the ceiling) with angle $\theta_{cone}$. Finally, the multi--rotor must avoid a known set of obstacles that can be either ellipsoids or polytopes. 
The $j^{th}$ ellipsoidal is expressed as $ x' A_j x + 2 b'_j x + c_j \leq 0 $, while
The $k^{th}$ polytopical is defined by its faces $ A_k x + b_k \leq 0 $.

Given these constraints, the objective is to find a minimum fuel trajectory from a prescribed $x(t_0)$ to a known $x(t_f)$ with fixed final time $t_f$ and $N$ discrete points along with $n_{ellip} + n_{polyt}$  obstacles to avoid:
\begin{problem}
	Minimum Fuel 3-DoF Multi--rotor Obstacle Avoidance
	\begin{equation*}
	\begin{aligned}
	\operatorname{min}& \; \sum_{i=1}^N \|u_i\|\\
	\textnormal{s.t. }& \; x_{i+1} = Ax_i + B(u_i+g),\ i=1,\ldots,N-1\\
	& \|u_i\|_2 \leq u_{\max},\ i = 1,\ldots,N,\\
	& \|v_i\|_2\leq V_{\max}, \ i = 1,\ldots,N,\\
	& \hat{n}^Tu_i \geq \|u_i\|_2\cos(\theta_{cone}), \ i = 1,\ldots,N,\\
	& x'_i A_j x_i + 2 b'_j x_i + c_j \geq 0,\ i=1,\ldots,N,\ j=1,\ldots,n_{ellip},\\
	& A_k x_i + b_k \geq 0, \ i=1,\ldots,N,\ k=1,\ldots,n_{polyt},\\
	& x_0 = x(t_0),\ x_N = x(t_f).
	\end{aligned}
	\end{equation*}
\end{problem}

\begin{table}
	\centering
	\caption{Parameter Values of the \fast algorithm test}
	\label{tab:pars}
	\begin{tabular}{lc|lc}
		\hline\hline
		Par. & Value &  Par. & Value\rule{0pt}{2.6ex} \\
		\hline
		$N$ & 20 & $\epsilon$ & $1\times 10^{-4}$\rule{0pt}{2.6ex}\\
		$V_{\max}$ & 2 m/s & $u_{\max}$ & 13.33 m/s$^2$\\
		$g$ & $\lbrack 0,\,0,\,-9.81\rbrack^T$ m/s$^2$ & $\theta_{cone}$ & 30 deg \\
		$p_0$ & $\lbrack -2,\,6,\,0\rbrack^T$ m & $p_f$ & $\lbrack 6,\,2,\,0.5\rbrack^T$ m\\
		$v_0$ & $\lbrack 0,\,0,\,0\rbrack^T$ m/s & $v_f$ & $\lbrack 0,\,0,\,0\rbrack^T$ m/s\\
		$\lambda$ & 0 & $\hat{n}$ & $\lbrack 0,\,0,\,1\rbrack^T$\\
		$t_f$ & 15 s &&\\
		\hline\hline
	\end{tabular}
\end{table}

The parameters given in Table \ref{tab:pars} are used to obtain the numerical results presented herein. An initial feasible trajectory is obtained by using Algorithm~\ref{algo:infeas-ini}. The infeasible initial trajectory $ z^{\{0\}} $ is shown in Figure~\ref{fig:initial_traj} (black), while the feasible and (locally) optimal converged trajectory $ \bar{z} $ is shown in Figure~\ref{fig:converged_traj} (green)


\begin{figure}[!h]
	\centering
	\includegraphics[width=.7\linewidth,trim=20 0 30 15,clip]{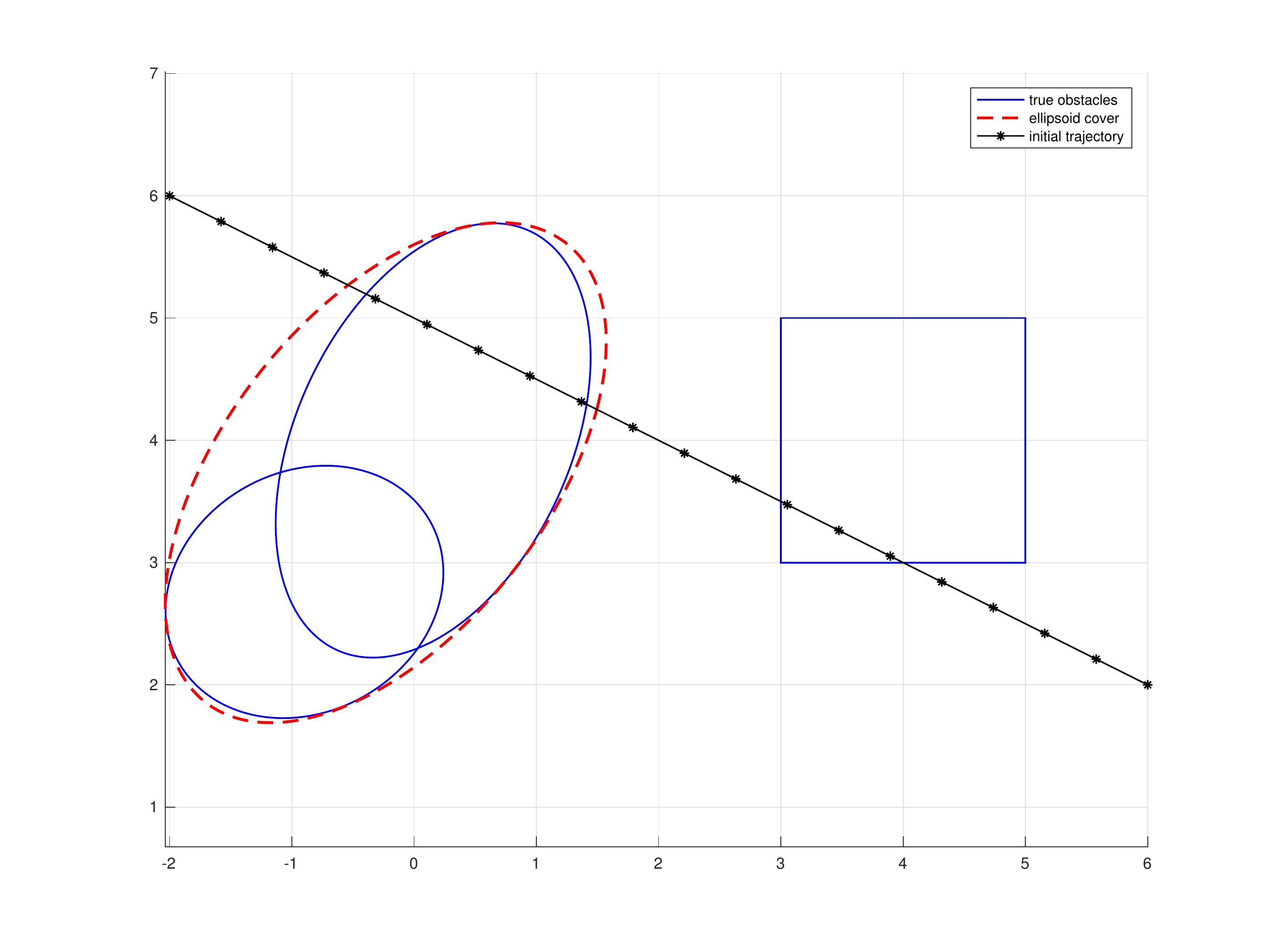}
	\caption{The \fast algorithm test: Initial trajectory (black) and the minimum volume ellipsoidal cover (red) for intersecting obstacles (blue).}
	\label{fig:initial_traj}
\end{figure}

\begin{figure}[!h]
	\centering
	\includegraphics[width=.7\linewidth,trim=20 0 30 15,clip]{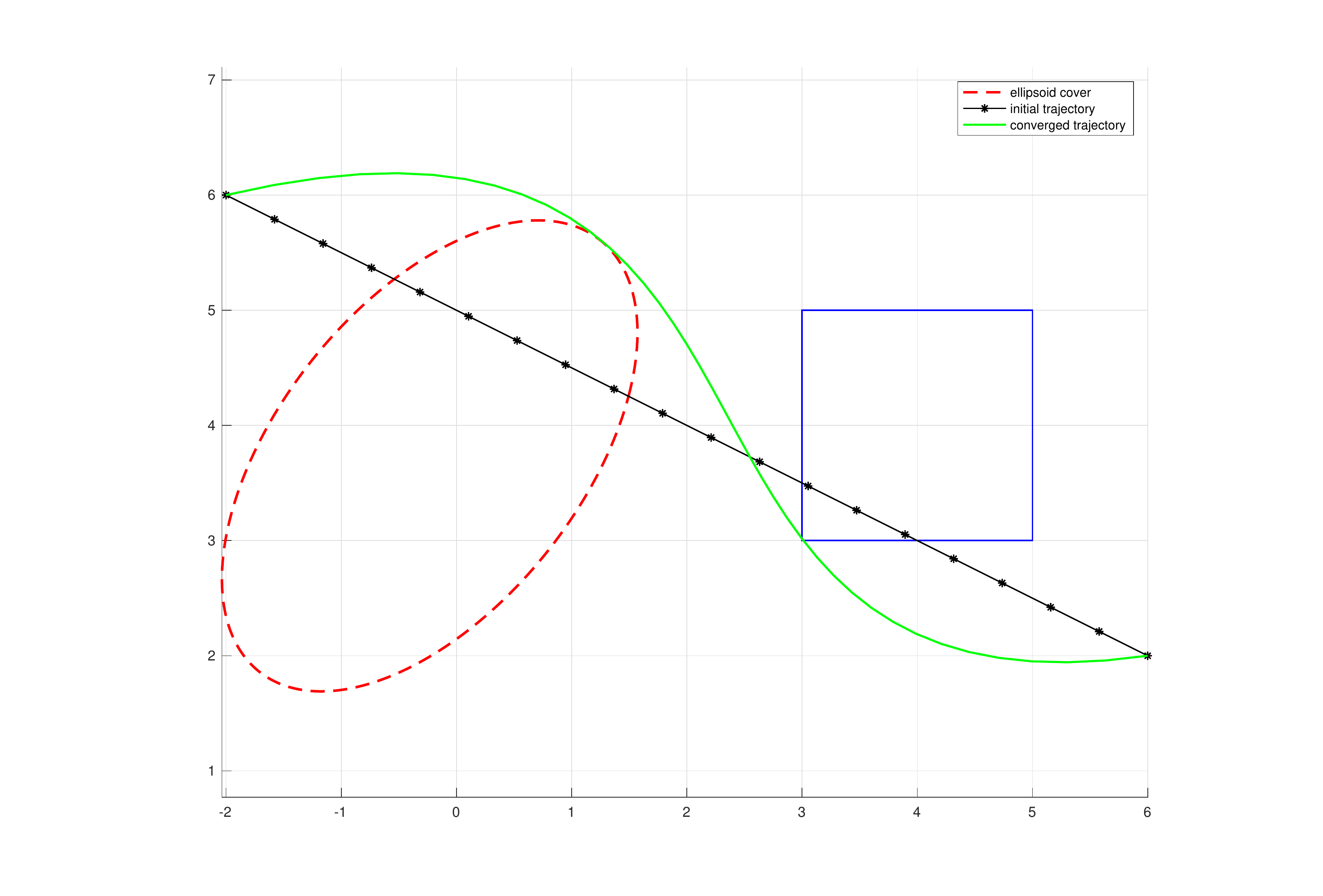}
	\caption{The \fast algorithm test: converged trajectory (green) and the minimum volume ellipsoidal cover (red) for intersecting obstacles (blue).}
	\label{fig:converged_traj}
\end{figure}

The \fast algorithm is initiated with $z^{\{0\}}$, and is considered to have converged when the improvement in the cost of the convexified problem is less than $\epsilon$. Figure \ref{fig:converged_traj} illustrates the converged trajectory that avoids the obstacles while satisfying its actuator and mission constraints (green). Note that the converged trajectory in Figure \ref{fig:converged_traj} (blue x's) is different than that of $z^{\{0\}}$ and is characterized by having a smooth curve. For $z^{\{0\}}$, $\sum_{i=1}^N \|u_i\| = 282.41$ and at the converged solution $z^{\{11\}}$, we have that $\sum_{i=1}^N \|u_i\| = 221.76$, so the cost of the converged trajectory is lower than that of the initial trajectory. At each iteration, the \fast algorithm solves an SOCP, and therefore 11 SOCPs were solved in order to produce these results. The full convergence history is shown in Figure~\ref{fig:convergence}, which clearly demonstrates the pattern of a superlinearly convergent algorithm: It improves relatively slowly at first, but significantly speeds up later on, especially when approaching the converged solution.

\begin{figure}[!h]
	\centering
	\includegraphics[width=.8\linewidth,trim=20 0 30 15,clip]{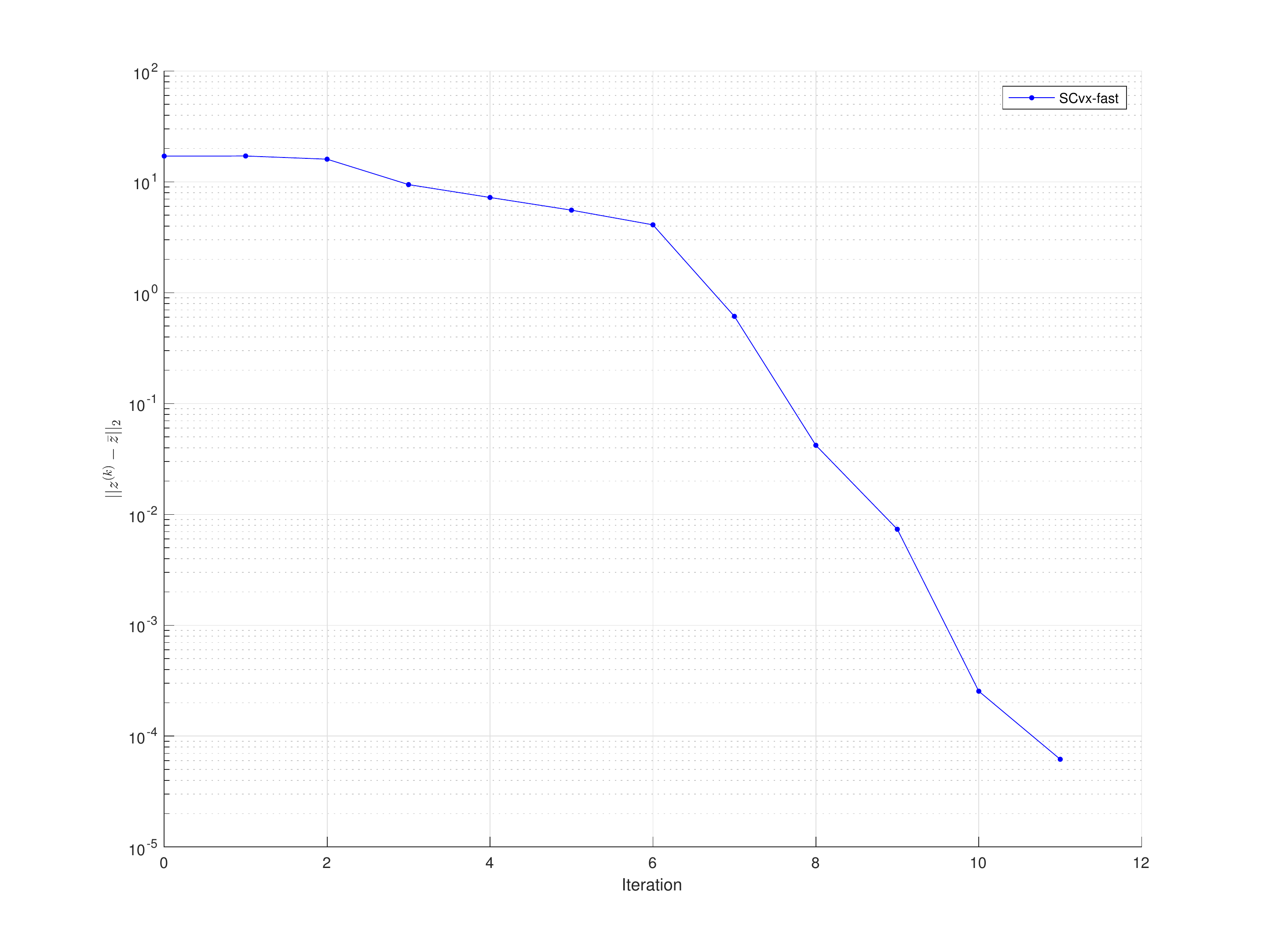}
	\caption{The \fast algorithm test: Full convergence history of the solution variable $ z $. It affirms a typical pattern for superlinear convergence rate, namely, speeding up when near the converged solution.}
	\label{fig:convergence}
\end{figure}


%

\section{Conclusions}\label{sec:conc}
This paper formally propose \fast, a new variant within the Successive Convexification algorithmic framework. The said algorithm solves non-convex optimal control problems with specific types of state constraints (i.e. union of convex keep-out zones) and is often faster to converge than \scvx, its predecessor as shown in the numerical results. In order to preserve more feasibility, the proposed algorithm uses a novel \pac procedure to successively convexify both state constraints and system dynamics, and thus a finite dimensional convex programming subproblem is solved at each succession. It also gets rid of the dependency on trust regions, gaining the ability to take larger steps and thus ultimately attaining faster convergence. The extension is in three folds as follows. i) We can now initialize the algorithm from an infeasible starting point, and regain feasibility in just one step; ii) We get rid of the smoothness conditions on the constraints so that a broader range of ``obstacles" can be included. Significant changes are made to adjust the algorithm accordingly; iii) We obtain a proof of superlinear rate of convergence. The \fast algorithm is particularly suitable for solving trajectory planning problems with collision avoidance constraints. Numerical simulations are performed and affirms the fast convergence rate. With powerful convex programming solvers, the algorithm can be implemented onboard for real-time autonomous guidance applications.

\bibliographystyle{abbrv}
\bibliography{scvx_fast}

\begin{thebibliography}{10}

\bibitem{behcet_aut11}
B.~{A\c c\i kme\c se} and L.~Blackmore.
\newblock Lossless convexification of a class of optimal control problems with
  non-convex control constraints.
\newblock {\em Automatica}, 47(2):341--347, 2011.

\bibitem{pointing2013}
B.~{A\c c\i kme\c se}, J.~Carson, and L.~Blackmore.
\newblock Lossless convexification of non-convex control bound and pointing
  constraints of the soft landing optimal control problem.
\newblock {\em IEEE Transactions on Control Systems Technology},
  21(6):2104--2113, 2013.

\bibitem{accikmese2006convex}
B.~{A\c c\i kme\c se}, D.~P. Scharf, E.~A. Murray, and F.~Y. Hadaegh.
\newblock A convex guidance algorithm for formation reconfiguration.
\newblock In {\em Proceedings of the AIAA Guidance, Navigation, and Control
  Conference and Exhibit}, 2006.

\bibitem{allen2016real}
R.~Allen and M.~Pavone.
\newblock A real-time framework for kinodynamic planning with application to
  quadrotor obstacle avoidance.
\newblock In {\em AIAA Guidance, Navigation, and Control Conference}, page
  1374, 2016.

\bibitem{augugliaro2012generation}
F.~Augugliaro, A.~P. Schoellig, and R.~D'Andrea.
\newblock Generation of collision-free trajectories for a quadrocopter fleet: A
  sequential convex programming approach.
\newblock In {\em IEEE/RSJ International Conference on Intelligent Robots and
  Systems (IROS)}, pages 1917--1922, 2012.

\bibitem{bertsimas2010robust}
D.~Bertsimas and O.~Nohadani.
\newblock Robust optimization with simulated annealing.
\newblock {\em Journal of Global Optimization}, 48(2):323--334, 2010.

\bibitem{lars_sys12}
L.~Blackmore, B.~{A\c c\i kme\c se}, and J.~M. Carson.
\newblock Lossless convexfication of control constraints for a class of
  nonlinear optimal control problems.
\newblock {\em System and Control Letters}, 61(4):863--871, 2012.

\bibitem{BoydConvex}
S.~Boyd and L.~Vandenberghe.
\newblock {\em Convex Optimization}.
\newblock Cambridge University Press, 2004.

\bibitem{buskens}
C.~Buskens and H.~Maurer.
\newblock Sqp-methods for solving optimal control problems with control and
  state constraints: adjoint variables, sensitivity analysis, and real-time
  control.
\newblock {\em Journal of Computational and Applied Mathematics}, 120:85--108,
  2000.

\bibitem{chen2015decoupled}
Y.~Chen, M.~Cutler, and J.~P. How.
\newblock Decoupled multiagent path planning via incremental sequential convex
  programming.
\newblock In {\em 2015 IEEE International Conference on Robotics and Automation
  (ICRA)}, pages 5954--5961. IEEE, 2015.

\bibitem{domahidi2013ecos}
A.~Domahidi, E.~Chu, and S.~Boyd.
\newblock {ECOS}: An{ SOCP} solver for embedded systems.
\newblock In {\em European Control Conference (ECC)}, pages 3071--3076. IEEE,
  2013.

\bibitem{dueri2014automated}
D.~Dueri, J.~Zhang, and B.~{A\c c\i kme\c se}.
\newblock Automated custom code generation for embedded, real-time second order
  cone programming.
\newblock In {\em 19th IFAC World Congress}, pages 1605--1612, 2014.

\bibitem{han1979exact}
S.-P. Han and O.~L. Mangasarian.
\newblock Exact penalty functions in nonlinear programming.
\newblock {\em Mathematical programming}, 17(1):251--269, 1979.

\bibitem{matt_aut14}
M.~Harris and B.~{A\c c\i kme\c se}.
\newblock Lossless convexification of non-convex optimal control problems for
  state constrained linear systems.
\newblock {\em Automatica}, 50(9):2304--2311, 2014.

\bibitem{hauser2006barrier}
J.~Hauser and A.~Saccon.
\newblock A barrier function method for the optimization of trajectory
  functionals with constraints.
\newblock In {\em Proceedings of the 45th IEEE Conference on Decision and
  Control}, pages 864--869. IEEE, 2006.

\bibitem{houska2011auto}
B.~Houska, H.~J. Ferreau, and M.~Diehl.
\newblock An auto-generated real-time iteration algorithm for nonlinear mpc in
  the microsecond range.
\newblock {\em Automatica}, 47(10):2279--2285, 2011.

\bibitem{hull1997}
D.~Hull.
\newblock Conversion of optimal control problems into parameter optimization
  problems.
\newblock {\em Journal of Guidance, Control, and Dynamics}, 20(1):57--60, 1997.

\bibitem{lipp2016variations}
T.~Lipp and S.~Boyd.
\newblock Variations and extension of the convex--concave procedure.
\newblock {\em Optimization and Engineering}, 17(2):263--287, Jun 2016.

\bibitem{liu2014solving}
X.~Liu and P.~Lu.
\newblock Solving nonconvex optimal control problems by convex optimization.
\newblock {\em Journal of Guidance, Control, and Dynamics}, 37(3):750--765,
  2014.

\bibitem{liu2015entry}
X.~Liu, Z.~Shen, and P.~Lu.
\newblock Entry trajectory optimization by second-order cone programming.
\newblock {\em Journal of Guidance, Control, and Dynamics}, 39(2):227--241,
  2015.

\bibitem{mao2017successive}
Y.~Mao, D.~Dueri, M.~Szmuk, and B.~{A\c c\i kme\c se}.
\newblock Successive convexification of non-convex optimal control problems
  with state constraints.
\newblock {\em IFAC-PapersOnLine}, 50(1):4063 -- 4069, 2017.

\bibitem{SCvx_cdc16}
Y.~Mao, M.~Szmuk, and B.~{A\c c\i kme\c se}.
\newblock Successive convexification of non-convex optimal control problems and
  its convergence properties.
\newblock In {\em IEEE 55th Conference on Decision and Control (CDC)}, pages
  3636--3641, 2016.

\bibitem{mattingley2012}
J.~Mattingley and S.~Boyd.
\newblock Cvxgen: A code generator for embedded convex optimization.
\newblock {\em Optimization and Engineering}, 13(1):1--27, 2012.

\bibitem{richards2002}
A.~Richards, T.~Schouwenaars, J.~P. How, and E.~Feron.
\newblock Spacecraft trajectory planning with avoidance constraints using
  mixed-integer linear programming.
\newblock {\em Journal of Guidance, Control, and Dynamics}, 25(4):755--764,
  2002.

\bibitem{rockafellar1972state}
R.~T. Rockafellar.
\newblock State constraints in convex control problems of bolza.
\newblock {\em SIAM journal on Control}, 10(4):691--715, 1972.

\bibitem{convex_lipschitz}
W.~S. U. M. D.~C. Room.
\newblock Every convex function is locally lipschitz.
\newblock {\em The American Mathematical Monthly}, 79(10):1121--1124, 1972.

\bibitem{rosen1966iterative}
J.~B. Rosen.
\newblock Iterative solution of nonlinear optimal control problems.
\newblock {\em SIAM Journal on Control}, 4(1):223--244, 1966.

\bibitem{rudin1964principles}
W.~Rudin.
\newblock {\em Principles of mathematical analysis}, volume~3.
\newblock 1964.

\bibitem{schulman2014motion}
J.~Schulman, Y.~Duan, J.~Ho, A.~Lee, I.~Awwal, H.~Bradlow, J.~Pan, S.~Patil,
  K.~Goldberg, and P.~Abbeel.
\newblock Motion planning with sequential convex optimization and convex
  collision checking.
\newblock {\em The International Journal of Robotics Research},
  33(9):1251--1270, 2014.

\bibitem{steinfeldt2010guidance}
B.~A. Steinfeldt, M.~J. Grant, D.~A. Matz, R.~D. Braun, and G.~H. Barton.
\newblock Guidance, navigation, and control system performance trades for mars
  pinpoint landing.
\newblock {\em Journal of Spacecraft and Rockets}, 47(1):188--198, 2010.

\bibitem{szmuk2016successive}
M.~Szmuk and B.~A. {A\c c\i kme\c se}.
\newblock Successive convexification for fuel-optimal powered landing with
  aerodynamic drag and non-convex constraints.
\newblock In {\em AIAA Guidance, Navigation, and Control Conference}, page
  0378, 2016.

\bibitem{szmuk2017successive}
M.~Szmuk, U.~Eren, and B.~{A\c c\i kme\c se}.
\newblock Successive convexification for mars 6-dof powered descent landing
  guidance.
\newblock In {\em AIAA Guidance, Navigation, and Control Conference}, page
  1500, 2017.

\bibitem{wulbert_project}
D.~Wulbert.
\newblock Continuity of metric projections.
\newblock {\em Transactions of the American Mathematical Society},
  134(2):335--341, 1968.

\bibitem{Zeilinger2014683}
M.~N. Zeilinger, D.~M. Raimondo, A.~Domahidi, M.~Morari, and C.~N. Jones.
\newblock On real-time robust model predictive control.
\newblock {\em Automatica}, 50(3):683 -- 694, 2014.

\end{thebibliography}

\end{document}